\documentclass{mystyle}
\usepackage{graphicx}
\usepackage{amsfonts, amssymb, amsmath, mathrsfs, latexsym}
\usepackage[all]{xy}
\usepackage[dvipsnames]{xcolor}
\usepackage{mathtools}
\usepackage{tikz}
\usepackage{tikz-cd}
\usepackage{cancel}
\tikzcdset{arrow style=tikz, diagrams={>=stealth}}      
\usetikzlibrary{positioning}
\usepackage[dvipsnames]{xcolor}
\usepackage[normalem]{ulem}
\usepackage{booktabs}

\ifpdf \usepackage[unicode,pdftex]{hyperref} \input glyphtounicode\pdfgentounicode=1
\else\usepackage[unicode,dvipdfm]{hyperref}\fi
\usepackage{pgfplots}
\usepgfplotslibrary{fillbetween}
\pgfplotsset{width=11.5cm, compat=1.8}

\newtheorem{thm}{Theorem}[section]
\newtheorem{prop}[thm]{Proposition}
\newtheorem{lem}[thm]{Lemma}
\newtheorem{cor}[thm]{Corollary}

\theoremstyle{definition}

\newtheorem{deff}[thm]{Definition}

\newtheorem{rmk}[thm]{Remark}

\newcommand{\enm}[1]{\ensuremath{#1}}          %
\newcommand{\op}[1]{\operatorname{#1}}
\newcommand{\cal}[1]{\mathcal{#1}}

\renewcommand{\bar}[1]{\overline{#1}}

\newcommand{\EE}{\enm{\mathbb{E}}}

\newcommand{\NN}{\enm{\mathbb{N}}}

\newcommand{\FF}{\enm{\mathbb{F}}}
\renewcommand{\AA}{\enm{\mathbb{A}}}
\newcommand{\PP}{\enm{\mathbb{P}}}

\newcommand{\Dd}{\enm{\cal{D}}}
\newcommand{\Ee}{\enm{\cal{E}}}
\newcommand{\Ff}{\enm{\cal{F}}}
\newcommand{\Gg}{\enm{\cal{G}}}
\newcommand{\Hh}{\enm{\cal{H}}}
\newcommand{\Ii}{\enm{\cal{I}}}

\newcommand{\Ll}{\enm{\cal{L}}}
\newcommand{\Mm}{\enm{\cal{M}}}

\newcommand{\Oo}{\enm{\cal{O}}}

\newcommand{\Ss}{\enm{\cal{S}}}

\newcommand{\Xx}{\enm{\cal{X}}}

\renewcommand{\phi}{\varphi}
\renewcommand{\theta}{\vartheta}
\renewcommand{\epsilon}{\varepsilon}

\newcommand{\Aut}{\op{Aut}}

\newcommand{\codim}{\op{codim}}

\renewcommand{\to}[1][]{\xrightarrow{\ #1\ }}

\newcommand{\GG}{\ensuremath{\mathbb{G}}}
\newcommand{\h}{\ensuremath{\mathcal}}
\newcommand{\HL}{\ensuremath{\mathcal{H}^\mathcal{L}_}}

\newcommand{\w}{\widetilde}
\newcommand{\ce}{C_\mathcal{E}}

\newcommand{\vni}{\vskip 4pt \noindent}
\newcommand{\al}{\ensuremath{\alpha}}

\title[Hilbert scheme of smooth curves of degree sixteen in $\PP^5$]
{Hilbert scheme of  smooth curves of degree sixteen in $\mathbb{P}^5$}
\thanks{
{ The author was supported in part by National Research Foundation of South Korea (2022R1I1A1A01055306).} }

\dedicatory{To  Edoardo Ballico,  with admiration and friendship}
\author[Changho Keem]{Changho Keem}
\address{
Department of Mathematics,
Seoul National University\\
Seoul 151-742,  
South Korea}
\email{ckeem1@gmail.com}
\date{\today}
\begin{document}
\begin{abstract}
We denote by $\mathcal{H}_{d,g,r}$ the Hilbert scheme of smooth curves, which is the union of components whose general point corresponds to a smooth irreducible and non-degenerate curve of degree $d$ and genus $g$ in $\mathbb{P}^r$. 
In this article, we study $\mathcal{H}_{16,g,5}$ for almost every possible genus $g$ and chasing after its irreducibility. We also study the natural moduli map $\Hh_{d,g,5}\stackrel{\mu}{\to}\Mm_g$
and several key properties such as  gonality of a general element as well as characterizing smooth elements in each component.

\end{abstract}
\maketitle		
\section{
An overview,  motivation and preliminary set up}

Let $\h{H}_{d,g,r}$ be the Hilbert scheme of smooth, irreducible and non-degenerate curves of degree $d$ and genus $g$ in $\PP^r$. 
We denote by $\mathcal{H}^\mathcal{L}_{d,g,r}$ the union of components of $\mathcal{H}_{d,g,r}$ whose general element corresponds to a {\it linearly normal} curve $C\subset \PP^r$.

In this paper we study Hilbert schemes $\Hh_{16,g,5}$ of smooth curves in $\PP^5$ of degree sixteen and genus $g$. 
Specifically, we determine when $\Hh_{16,g,5}$ is irreducible and study the moduli map
$\Hh_{d,g,5}\stackrel{\mu}{\to}\Mm_g$ for almost all possible genus $0\le g\le 21$. We also study key properties such as  gonality as well as characterizing smooth elements in each component.

In general, determining the irreducibility of Hilbert schemes is rather a non-trivial task, which goes back to Severi \cite{Sev} who asserted with an incomplete proof that the Hilbert scheme  $\h{H}_{d,g,r}$ is irreducible for 
the triples of $(d,g,r)$ in the range 
\vskip 8pt

(i) $d\ge g+r~~~~$ or 

\noindent
\vskip 8pt
\noindent
in the following  Brill-Noether range which is much wider

\vskip 12pt
(ii) $\rho(d,g,r):=g-(r+1)(g-d+r)\ge 0$.  

\bigskip
The assertion of Severi turns out to be true for $r=3, 4$ under the condition (i); cf. \cite{E1, E2}. It is  also known that $\h{H}_{d,g,3}$ is 
irreducible in an extended range $d\ge g$; cf. \cite{KKy1} and references therein. 
For $r=4$, irreducibility of $\HL{d,g,4}$ also holds in the range $d\ge g+1$ except for some sporadic small values of the genus $g$; cf. \cite{I,KK3, KKy2}.

For $r=5$ the irreducibility of $\h{H}_{d,g,5}$ is not known in general in the range $d\ge g+5$ as conjectured by Severi. The best known result so far regarding the irreducibility of $\h{H}_{d,g,5}$ is the result of H. Iliev who showed  that
$\h{H}_{d,g,5}$ is irreducible whenever $d\ge \max\{\frac{11}{10}g+2,g+5\}$; cf.
\cite{I2}. 

There have been attempts for a better understanding of $\h{H}_{d,g,5}$ when the
degree of projective curves in question is relatively low. For example, 
in recent works of the author jointly with E. Ballico, $\h{H}_{d,g,5}$ has been studied 
rather extensively for $d\le 15$; cf. \cite{bumi}, \cite{edinburgh}.
In this paper, we push forward one step further to the case $d=16$, which is the main object of our study and motivation. 


\noindent

\subsection{Notations and conventions}
For notation and conventions, we  follow those in \cite{ACGH} and \cite{ACGH2}; e.g. $\pi (d,r)$ is the maximal possible arithmetic genus of an irreducible,  non-degenerate and reduced curve of degree $d$ in $\PP^r$ which is usually referred as the first Castelnuovo genus bound. We shall refer to curves $C\subset\PP^r$ of degree $d$ whose (arithmetic) genus equals $\pi(d,r)$ as {\it extremal curves}. $\pi_1(d,r)$ is the so-called the second Castelnuvo genus bound which is the maximal possible arithmetic genus of  an irreducible, non-degenerate and reduced curve of degree $d$ in $\PP^r$ not lying on a  surface of minimal degree $r-1$; cf. \cite[page 99]{H1}, \cite[page 123]{ACGH}.
We shall call curves $C\subset\PP^r$ of degree $d$ and (arithmetic) genus $g$ such that $\pi_1(d,r)<g\le\pi(d,r)$ {\it nearly extremal curves}.

Following classical terminology, a linear series of degree $d$ and dimension $r$ on a smooth curve $C$ is denoted by $g^r_d$.
A base-point-free linear series $g^r_d$ ($r\ge 2$) on a smooth curve $C$ is called {\it birationally very ample} when the morphism 
$C \rightarrow \mathbb{P}^r$ induced by  the $g^r_d$ is generically one-to-one onto (or is birational to) its image curve.
A base-point-free linear series $g^r_d$ on $C$  is said to be compounded of an involution ({\it compounded} for short) if the morphism induced by the linear series $g^r_d$ gives rise to a non-trivial covering map $C\rightarrow C'$ of degree $k\ge 2$. 
Throughout we work exclusively over the field of complex numbers.

\subsection{Preliminary set up}
\vni
We briefly recall a couple of fundamental results and basic frameworks for our study which are  well-known; cf. \cite{ACGH2}  or \cite[\S 1 and \S 2]{AC2}.

Let $\mathcal{M}_g$ be the moduli space of smooth curves of genus $g$. Given an isomorphism class $[C] \in \mathcal{M}_g$ corresponding to a smooth irreducible curve $C$, there exist a neighborhood $U\subset \mathcal{M}_g$ of the class $[C]$ and a smooth connected variety $\mathcal{M}$ which is a finite ramified covering $h:\mathcal{M} \to U$, as well as  varieties $\mathcal{C}$, $\mathcal{W}^r_d$ and $\mathcal{G}^r_d$ proper over $\mathcal{M}$ with the following properties:
\begin{enumerate}
\item[(1)] $\xi:\mathcal{C}\to\mathcal{M}$ is a universal curve, i.e. for every $p\in \mathcal{M}$, $\xi^{-1}(p)$ is a smooth curve of genus $g$ whose isomorphism class is $h(p)$,
\item[(2)] $\mathcal{W}^r_d$ parametrizes the pairs $(p,L)$ where $L$ is a line bundle of degree $d$ and $h^0(L) \ge r+1$ on $\xi^{-1}(p)$,
\item[(3)] $\mathcal{G}^r_d$ parametrizes the couples $(p, \mathcal{D})$, where $\mathcal{D}$ is possibly an incomplete linear series of degree $d$ and dimension $r$ on $\xi^{-1}(p)$.
\end{enumerate}

\noindent
For a complete linear series $\h{D}$ on a smooth curve $C$, the residual series $|K_C-\h{D}|$ is sometimes denoted by $\h{D}^\vee$.
{Given an irreducible family $\h{F}\subset\h{G}^r_d$ with some geometric meaning, whose general member is complete,  the closure of the family $\{\h{D}^\vee| \h{D}\in \h{F}, ~\h{D} \textrm{ is complete}\}\subset \h{W}^{g-d+r-1}_{2g-2-d}$ is  denoted by $\h{F}^\vee$.

\noindent
Let $\boldmath{\widetilde{\mathcal{G}}}$ ($\boldmath{\widetilde{\mathcal{G}}_\mathcal{L}}$,  respectively) be  the union of components of $\mathcal{G}^{r}_{d}$ whose general element $(p,\mathcal{D})$ corresponds to a very ample (very ample and complete, respectively) linear series $\mathcal{D}$ on the curve $C=\xi^{-1}(p)$. By recalling that an open subset of $\mathcal{H}_{d,g,r}$ consisting of elements corresponding to smooth irreducible and non-degenerate curves is a $\Aut(\PP^r)$-bundle over an open subset of $\widetilde{\mathcal{G}}$, {\it the irreducibility of $\boldmath{\widetilde{\mathcal{G}}}$ guarantees the irreducibility of $\boldmath{\mathcal{H}_{d,g,r}}$. Likewise, the irreducibility of $\boldmath{\widetilde{\mathcal{G}}_\mathcal{L}}$ ensures the irreducibility of 
 $\boldmath{\mathcal{H}_{d,g,r}^\mathcal{L}}$.}

\noindent
\vni
We recall the following  fundamental fact regarding the scheme $\mathcal{G}^{r}_{d}$ which is also well-known; cf. \cite[2.a]{H1} and \cite[Ch. 21, \S 3, 5, 6, 11, 12]{ACGH2}. 
\begin{prop}\label{facts}
For non-negative integers $d$, $g$ and $r$, let $$\rho(d,g,r):=g-(r+1)(g-d+r)$$ be the Brill-Noether number.
The dimension of any component of $\mathcal{G}^{r}_{d}$ is at least $$\lambda(d,g,r):=3g-3+\rho(d,g,r), $$ 
hence the dimension of any component of $\h{H}_{d,g,r}$ is at least
$$\h{X}(d,g,r):=\lambda (d,g,r)+\dim\Aut(\PP^r).$$ Moreover, if $\rho(d,g,r)\ge 0$, there exists a unique component $\mathcal{G}_0\subset\widetilde{\mathcal{G}}$ which dominates $\mathcal{M}$(or $\mathcal{M}_g$).
	\end{prop}

\vni
\subsection{Organization of the paper}
In the next section we list up several auxiliary results and remind readers a couple of easy and elementary lemmas which are necessary for our study. 
In the third section, which is divided into several subsections, we prove our main results; determination of  the irreducibility or the reducibility of a certain non-empty $\mathcal{H}^\mathcal{L}_{16,g,5}$ for $g\neq 15,16$ as well as characterization of a general element in each component.

{

\section{Some relevant generalities and auxiliary results}
In this section we prepare several facts which are relevant to our study.  Almost all the results we quote in this section are well known and we list them up mainly for fixing notation.
\subsection{Severi variety of nodal curves on Hirzebruch surfaces}
We begin to recall the following generalities regarding the Severi variety of (nodal) curves on a  Hirzebruch surface $\mathbb{F}_e=\PP (\h{O}_{\PP^1}\oplus\h{O}_{\PP^1}(-e))\stackrel{\pi}{\rightarrow}\PP^1$, i.e. a geometrically ruled surface over $\PP^1$ with invariant $e\ge 0$. 
\begin{deff}
\begin{enumerate}
\item[(i)]
Given a Hirzebruch surface $\mathbb{F}_e$, let $C_0$ be the section with $C_0^2=-e\le 0$ and $f$ be a fibre of $\pi$.  Given a very ample $\h{L}=|aC_0+bf|$ on $\mathbb{F}_e$, let
 $p_a(\h{L})=(a-1)(b-1-\frac{1}{2}ae)$
 be the arithmetic genus of an integral curve belonging to $\h{L}$.
\item[(ii)]
Given an integer $0\le g\le p_a(\h{L})$, we set
 $\delta =p_a(\h{L})-g$. We denote by $$\Sigma_{\h{L}, \delta}\subset\PP(H^0(X,\h{L}))$$ the (equi-singular) Severi variety which is the closure of the locus of integral curves in $\h{L}$ whose singular locus consists of only $\delta$ nodes and no further singularities.
\item[(iii)]
Let $\Sigma_{\h{L},g}$ be the  (equi-generic) Severi variety which is the closure of the locus of integral curves of geometric genus $g$  in $\h{L}$. 
\end{enumerate}
\end{deff}
\begin{rmk}\label{Severi}
\begin{enumerate}
\item[(i)] A general member of every irreducible component of the equi-generic Severi variety $\Sigma_{\h{L},g}$ is a nodal curve; cf. \cite[Proposition 2.1]{H2}, \cite[Theorem B2, p.177]{DS} and \cite[pp. 105-117]{H3}.  
\item[(ii)] The equi-singular Severi variety $\Sigma_{\h{L}, \delta}$ is {\it irreducible} of the expected dimension $\dim|\h{L}|-\delta
$
if non-empty; cf. \cite[Proposition 2.11, Theorem 3.1]{tyomkin}.
\end{enumerate}
\end{rmk}
For dimension count of certain families of curves under our study, we recall 
 the following and fix notation regarding  a surface $S\subset \PP^{n+1}$ with minimal degree $\deg S=n\ge 2$.

\begin{rmk}\label{minimal} For an irreducible and  non-degenerate surface $S\subset\PP^{n+1}$ of degree $n$, one of the following  holds; cf. \cite[IV, Exercises pp 53-54]{Beauville}.
\begin{itemize}
\item[(i)]
$S$ is a smooth rational normal surface scroll;
\begin{equation}\label{lsdimension}
\dim |aH+bL| =\frac{1}{2}a(a+1)n+(a+1)(b+1)-1
\end{equation}
where $H$ (resp. $L$) is the class of a hyperplane section (resp. the class a line of the ruling).
For any integral curve $C\in |aH+bL|$ with  $\deg C=d$,
\begin{equation}\label{scrolldegree}
d=na+b, ~~p_a(C)
=\frac{1}{2}\,a \left( a-1 \right)\cdot n  + \left( a-1 \right) 
 \left( b-1 \right).
\end{equation}
Whenever we deal with a rational normal surface scroll $S\subset\PP^{n+1}$, $H$ (resp. $L$) always denotes the class of a hyperplane section (resp. the class of a line of the ruling). 
\item[(ii)] $S$ is  a Veronese surface in $\PP^5$.
\item[(iii)]
$S$ is a cone over a rational normal curve in a hyperplane $\PP^{n}\subset\PP^{n+1}$.
\end{itemize}
\end{rmk}
\vni

The following is known widely as a folklore without explicit source of a proof known to  the author. 
Readers may consult \cite[Theorem 2.12]{Nasu} whose proof can be adopted easily to our current situation.

\begin{prop}\label{specialization} Smooth curves in $\PP^{n+1}$  lying on a cone over a rational normal curve in  a hyperplane $H\cong\PP^{n}$ is  a specialization of curves lying on a smooth rational normal surface scroll.
\end{prop}
The following proposition is of a similar kind, whose proof can be found in the work jointly with E. Ballico \cite[Proposition 2.1]{edinburgh}.
\begin{prop}\label{b1}
Let $S\subset \PP^r$, $3\le r \le 9$, be a normal rational del Pezzo surface of degree $r$. Fix an integral curve $C\subset S$. Then $C$ is a flat limit of a family of curves, 
contained in  smooth del Pezzo surfaces of degree
$r$.
\end{prop}

\subsection{Linear systems on  $k$-gonal curves and Castelnuovo-Severi inequality}
The following two lemmas -- which are basically of the same character -- will be used when we determine the very-ampleness  of the residual series of a multiple of a unique $g^1_k$  on a general $k$-gonal curve.
\begin{lem}\cite[Proposition 1.1]{CKM}\label{kveryample} Assume $2k-g-2<0$. Let $C$ be a general $k$-gonal curve of genus $g$, $k\ge 2$, $0\le m$, $n\in\mathbb{Z}$ such that 
\begin{equation}\label{veryamplek}
g\ge 2m+n(k-1)
\end{equation}
 and let $D\in C_m$. Assume that there is no $E\in g^1_k$ with $E\le D.$ Then $\dim|ng^1_k+D|=n$.
\end{lem}
\begin{lem}\cite[Prop. 1]{B1}\label{kveryample11}
Let $g, k$ be positive integers such that $k\ge 3, g\ge 2k -2$. Let $C$ be a general $k$-gonal curve with the unique pencil
$E$ of degree $k$.
For  $1\le l \le [ \frac{g}{k-1}],$
$\dim|lE|=l$. If $1\le l \le [ \frac{g}{k-1}]-2,$  $|K_C - lE|$ is very ample. \end{lem}

As a matter of fact, Lemma \ref{kveryample} is somewhat stronger than Lemma \ref{kveryample11}, which is more convenient to use.
The following inequality -- known as Castelnuovo-Severi inequality -- shall be used occassionally; cf. \cite[Theorem 3.5]{Accola1}.
\begin{rmk}[Castelnuovo-Severi inequality]\label{CS} Let $C$ be a curve of genus $g$ which admits coverings onto curves $E_h$ and $E_q$ with respective genera $h$ and $q$ of degrees $m$ and $n$ such that these two coverings admit no common non-trivial factorization; if $m$ and $n$ are primes this will always be the case. Then
$$g\le mh+nq+(m-1)(n-1).$$ 
\end{rmk}

We use the following simple criteria several times throughout the paper when we need to determine if the residual series of a pull-back of a certain linear series on a multiple covering is not very ample. The proof is easy and is left to readers.
\begin{lem}\label{easylemma1} Let $\h{E}=g^{\al -1}_e$ ($\al\ge 4$) be a special complete linear series on a smooth curve $C$ of genus $g\ge 5$, with (possibly empty) base locus $\Delta$. Set $\h{E}':=\h{E}-\Delta=g^{\al -1}_{e'}$. 
We assume that  either 
\begin{itemize}
\item [(i)]
$\h{E}'$ induces a morphism $C\stackrel{\eta}{\rightarrow} E\subset\PP^{\alpha-1}$, $\deg\eta=2$, onto $E$ of genus $h\ge 1$ with the normalization   $\w{E}\stackrel{\epsilon}{\rightarrow} E$. Let 
$C\stackrel{\w{\eta}}{\rightarrow} \w{E}$ be the morphism 
such that $\eta=\epsilon\circ\w{\eta}$. 
Suppose $|\epsilon^*(\h{O}_{\PP^{\al -1}}(1))|$
is {\it non-special} or


\item [(ii)] $\Delta\neq\emptyset$ and $\h{E}$ induces a triple covering onto a rational curve.
\end{itemize}
Then $\h{E}^\vee$ is not very ample. 
\end{lem}
\subsection{On moduli map}
Let $\mu: \Hh_{16,g,5}\to \Mm_g$ denote the natural functorial map - which we call the {\it moduli map} - sending $X\in\Hh_{16,g,5}$ to its isomorphism class
$\mu (X)\in\h{M}_g$.

\begin{rmk}\label{trivia1}
Let  $X\subset \PP^r$, $r\ge 2$, be a smooth curve of genus $g\ge 2$. Since $X$ has only finitely many automorphisms, the set $G:=\{h\in \Aut(\PP^r)| h(X)=X\}$ is a finite group. Hence the set $$\Aut(\PP^r)X:=\{Y\subset\PP^r | Y=\sigma (X) \mathrm{~ for ~some~ }\sigma\in \Aut (\PP^r)\}$$ consisting of all curves $Y\subset \PP^r$ projectively equivalent to $X$ is an irreducible  quasi-projective variety isomorphic to $\Aut(\PP^r)/G$. Thus $$\dim\Aut(\PP^r)X=\dim\Aut(\PP^r)=(r+1)^2-1$$ and $\dim \mu(\Hh) \le \dim \Hh -35$ for every irreducible family $\Hh \subseteq \Hh_{16,g,5}$.
\end{rmk}

\section{Curves of degree $16$ in $\PP^5$}
\subsection{Irreducibility, description of general element of $\Hh_{16,g,5}$ and the moduli map for $19\le g\le 21$}

By Castelnuovo genus bound, $\pi(16,5)=21$ and 
$\pi_1(16,5)=18$. In this subsection, we first determine the irreducibility of $\Hh_{16,g,5}$ with $19\le g\le 21$. We then proceed to describe a general element in each component of $\Hh_{16,g,5}$ and the fibre of the moduli map. The following is our first main result.

\begin{thm} 
\begin{itemize}
\item[(i)] $\Hh_{16,21,5}=\HL{16,21,5}$ is reducible with two components $\Hh_i$, $i=1,2$; 
\vni
\begin{enumerate}
\item[(a)] $\Hh_1$ generically consists of smooth plane curves of degree eight embedded into $\PP^5$ by the Veronese map, 

\vni
\item[(b)] the other one $\Hh_2$ consisting of $4$-gonal curves lying on a rational normal surface scroll. A general element in the second component  is isomorphic to a smooth curve of type $(4,8)$ on a smooth quadric $Q\subset\PP^3$ embedded into $\PP^5$ by the very ample 
linear system $|\Oo_Q(1,2)|$.

\vni
\item[(c)]
 $\mu^{-1}\mu(C)=\Aut(\PP^5)C$ for any smooth $C\in\h{H}_i$, $i=1,2$.
 \end{enumerate}
 \vni
\item[(ii)] $\Hh_{16,20,5}=\HL{16,20,5}$ is irreducible, a general $C\in\Hh_{16,20,5}$ is pentagonal,  $\mu^{-1}\mu(C)=\Aut(\PP^5)C$ and $\codim_{\Mm_g}\mu(\Hh_{16,20,5})=22$.
\item[(iii)] $\Hh_{16,19,5}=\emptyset$.
\end{itemize}
\end{thm}
\begin{proof} For the range $16\le g\le 21$ of the genus $g$, we note that $\Hh_{16,g,5}=\HL{16,g,5}$ since $\pi(16,6,5)=15$.  In range $19\le g\le 21$, it is known that  a smooth curve $C$ with $(d,g)=(16,g)$ lies on a surface $S\subset\PP^5$ of minimal degree $\deg S=4$; cf. \cite[Theorem 2.7, page 123]{ACGH}.

 By Remark \ref{minimal}, a well-known description of surfaces of minimal degree in $\PP^5$, we start to consider the following cases for the surface $S$:

(a) $S$ is a Veronese surface.

(b) $S$ is a (smooth) rational normal surface scroll.

(c) $S$ is a cone over a rational normal curve in $\PP^4$.

\begin{itemize}
\item[(i)] We assume $g=21$.

\vni
(a) If $S$ is a Veronese surface, $C$ is the image of a smooth plane curve of degree $8$
embedded by the Veronese embedding $\PP^2\hookrightarrow S\subset\PP^5$.
The family $\Hh_1\subset\Hh_{16,21,5}$ consisting of the image of smooth plane octics in the Veronese surface is an irreducible family with
\begin{align*}\dim\Hh_1&=\dim\PP(H^0(\PP^2,\Oo(8)))-\dim\Aut(\PP^2)+\dim\Aut(\PP^5)\\&=71>\Xx(16,21,5)=56.
\end{align*}
A general element $C\in\Hh_1$ is $7$-gonal with infinitely may $g^1_7$'s. 

\vni
(b) If $S$ is a rational normal surface scroll with $H$ the class of a hyperplane section and $L$ the class of the ruling of $S$, we set $C\in|aH+bL|$. 
We solve \eqref{scrolldegree} for $n=4$, $d=16$ and $p_a(C)=21$ to get $(a,b)=(4,0)$.
Such a family $\Hh_2$ consisting curves on rational normal surface scrolls form an irreducible family with 
$$\dim\Hh_2=\dim|aH+bL|+\dim\Ss(5)=44+29=73>\Xx(16,21,5).$$
by \eqref{lsdimension}
and
\begin{equation} \dim\Ss(r)=(r+3)(r-1)-3,
\end{equation}
where $\Ss(r)$ is the irreducible family parametrizing rational normal surface scrolls in $\PP^r$; cf. \cite[page 91]{H1}.
A general member of $\Hh_2$ is $4$-gonal with the unique $g^1_4$ cut out by rulings of the scroll by the Castelnuovo-Severi inequality.

\vni
(c) If $S$ is a cone over a rational normal curve in a hyperplane $\PP^4\subset\PP^5$, we recall the following generalities concerning curves on a cone $S\subset\PP^{n+1}$ over a rational normal curve in $\PP^n$, 
which is the image of the Hirzebruch surface $\mathbb{F}_n=\mathbb{P}(\h{O}_{\PP^1}\oplus\h{O}_{\PP^1}(-n))$. We denote by  $h$ (resp. $f$ ) the class in $\textrm{Pic}(\mathbb{F}_n)$ of the tautological bundle $\h{O}_{\mathbb{F}_n}(1)$ (resp. of a fibre); $f^2=0, f\cdot h=1, h^2=n, h=C_0+nf$.  The morphism $\mathbb{F}_n\rightarrow S\subset\PP^{n+1}$ given by the complete linear system $|h|$ is an embedding outside $C_0$ -- the curve with  negative self-intersection -- and contracts $C_0$ to the vertex $P$ of the cone $S$; $$C_0^2=-n, C_0\in |h-nf|, K_{\mathbb{F}_n}=-2h+(n-2)f=-2C_0+(-n-2)f.$$ Let $C\subset S$ be a reduced and non-degenerate curve of degree $d$ and let $\w{C}$ be the strict transformation of $C$ under the desingularization $\widetilde{S}\cong\mathbb{F}_n\rightarrow S$. 
~Setting $k=\w{C}\cdot f$, we have $\w{C}\equiv kh+(d-nk)f$. The adjunction formula gives
\begin{equation}\label{cone}
p_a(\w{C})=1/2\, \left( k-1 \right)  \left( 2\,d-nk-2 \right).
\end{equation}
We  further remark that 
\begin{equation}\label{conevertex}
0\le \w{C}\cdot C_0=\w{C}\cdot (h-nf)=d-nk = m
\end{equation} where $m$ is the multiplicity of $C$ at the vertex  $P$ of the cone $S$.
In our current situation, we have $d=16$, $p_a(\w{C})=g=21$, $k=4$ and $m=0$. Thus we may have a further family of smooth curves with $(d,g)=(16,21)$ lying on a cone $S$ which
may form an irreducible family $\Hh_3$ with
\begin{align}\label{H3}
\dim\Hh_3&=\dim\PP^5+\dim\Hh_{4,0,4}+\dim|4h|=5+\Xx(4,0,4)+\dim|4h|\nonumber\\&=5+21+44=70>\Xx(16,21,5),
\end{align}
where the first summand is the degree of freedom of the choice of the vertex of $S$ in $\PP^5$, the second summand is the number of parameters of rational normal curves $R$ in a complementary  hyperplane $H\cong\PP^4$ -- which is 
$$\dim\Hh_{4,0,4}=5\cdot H^0(\PP^1,\Oo(4))-1-\dim\Aut(\PP^1)=21$$ -- and the last summand is $\dim|\w{C}|$ on
$\FF_4$. The inequality \eqref{H3} suggests that  it may a priori be possible that $\Hh_3$ may form an irreducible component
of $\Hh_{16,21,5}$ different from $\Hh_2$. However, $\Hh_3$ is in the boundary of $\Hh_2$ which follows from a classical fact that smooth curves on a rational normal cone
are specializations of curves on  rational normal scrolls which we already mentioned; cf.   
Proposition \ref{specialization}.

Since a general member in $\Hh_1$ is $7$-gonal and a general member in $\Hh_2$ is $4$-gonal, whereas $\dim\Hh_1<\dim\Hh_2$, $\Hh_1$ is not in the boundary of $\Hh_2$ by lower semi-continuity of gonality.
Since we have exhausted all the possibilities and therefore $\Hh_1$ and $\Hh_2$ are the only two irreducible components of $\Hh_{16,21,5}$.

It is worthwhile to remark that without invoking the lower semi-continuity of gonality, it is possible to see that 
$\Hh_1$ is not in the boundary of $\Hh_2$ as follows. 

For a smooth curve $C\subset S\subset\PP^5$  of even degree $d=2a$
lying on a Veronese surface $S$,  we recall \cite[Proposition 2.4]{CC}\footnote{There is a minor type in the statement of \cite[Proposition 2.4]{CC}.}
\begin{equation}\label{plane}h^1(N_{C,\PP^5})=3(g-d+5)-3a+9.
\end{equation}
If $d=2a=16$, we have $h^1(N_{C,\PP^5})=15$. Denoting $\Xx(N_{C,\PP^r})$ by the Euler-Poincare characteristic of the normal bundle $N_{C,\PP^r}$ of $C\subset\PP^r$, we have
$$\Xx(d,g,r)=\Xx(N_{C,\PP^r})=h^0(N_{C,\PP^r})-h^1(N_{C,\PP^r})$$
and hence we obtain 
$$\dim T_C\Hh_1=h^0(N_{C,\PP^5})=\dim\Hh_1.$$
Thus $\Hh_1$ is indeed a generically reduced but super abundant component in the 
sense of Sernesi \cite[page 455]{Sernesi}.
Since
a smooth plane curve of degree $a\ge 4$ has a unique $g^2_a$ and a unique $g^5_{2a}$, $\mu^{-1}\mu(C)=\Aut(\PP^5)C$ for any smooth $C\in\h{H}_1$ and 
$$\mu(\Hh_1)\stackrel{bir}{\cong}\PP(H^0(\PP^2, \Oo(8))/\Aut(\PP^2), ~~\dim\mu(\Hh_1)=36.$$

We recall that there are two types of non-isomorphic rational normal surface scrolls in $\PP^5$. One 
is $\FF_0\cong\PP^1\times\PP^1=Q\subset\PP^3$ embedded by $|\Oo_Q(1,2)|$ as a quartic surface into $\PP^5$. The other one is the Hirzebruch surface $\FF_2$ embedded by $|C_0+3f|$ where $C_0^2=-2$, $f$ a fibre. Recall that the image of $\FF_2$ in $\PP^5$ are flat limits of the image of $\FF_0$ and this phenomenon is carried over to the the curves lying on them. Furthermore $\dim\Aut(\FF_0)=6<\dim\Aut(\FF_2)=7$ and a curve in $\FF_0$ is general. Let $C\in\Hh_2$ be a general element.  Hence $C\in|4H|=|\Oo_Q(1,2)^{\otimes 4}|=|\Oo_Q(4,8)|$
is $4$-gonal with a unique $g^1_4$ cut out by the ruling $|\Oo_Q(0,1)|$. The pencil $g^1_8$ cut out by the other ruling $|\Oo_Q(1,0)|$ is obviously base-point-free and complete since  a general $A\in g^1_8$ formed by distinct co-linear points imposes exactly seven conditions on 
$|\Oo_Q(2,6)|=|K_Q+C|$. Conversely, a finite set $A\subset C\subset Q$ with $\#A=8$ fails to 
impose independent conditions on $|\Oo_Q(2,6)|$ only if there exists $R\in|\Oo_Q(1,0)$ such that $A\subset R$. Therefore $g^1_8$ is unique and so is $|g^1_4+g^1_8|=g^3_{12}$ implying that the hyperplane series $g^5_{16}$ is unique. Thus $\mu^{-1}\mu(C)=\Aut(\PP^5)C$
for a general $C\in\Hh_2$. In a similar way, one can also show that for a smooth $C\subset\FF_2$, $\mu^{-1}\mu(C)=\Aut(\PP^5)C$. We have $$\dim\mu(\Hh_2)=\dim\Hh_2-\dim\Aut(\PP^5)=38.$$

\end{itemize}
\begin{itemize}
\item[(ii)] We assume $g=20$. 
\vni (a) There is no smooth curve with $(d,g)=(16, 20)$ in $\PP^5$ lying on a Veronese surface since $\binom{e-1}{2}\neq g=20$ for any $e\in\NN$.
\vni
(b) Assume $C\subset S$ is a rational normal scroll. We solve \eqref{scrolldegree} with $p_a(C)=20$ to get
$C\in|5H-4L|$. By \eqref{lsdimension}, the irreducible family $\Hh_1$ consisting of curves on surface scrolls has dimension
$$\dim\Hh_1=\dim|5H-4L|+\dim\Ss(5)=70>\Xx(16,20,5)=58.$$
\vni
(c) Assume $C\subset S$ is a rational normal cone. By \eqref{cone} and \eqref{conevertex}, there is no smooth curve with $(d,g)=(16,20)$ on a rational normal cone.
\vni

Therefore it follows that $\Hh_{16,20,5}$ is irreducible with the only component $\Hh_1$.
A 
general curve $$C\in|5H-4L|=|5(1,2)-4(0,1)|=|(5,6)|$$ in $\FF_0\cong Q$ is a smooth curve of type $(5,6)$ on a quadric, forming an irreducible family of dimension
$$\dim|\Oo_Q(5,6)|-\dim\Aut(Q)+\dim\Aut(\PP^5)=70.$$
\vni
A general $$C\in|5H-4L|=|5(C_0+3f)-4f|=|5C_0+11f|$$ lying on $\FF_2$ forms
an irreducible family of dimension 
$$\dim|5C_0+11f|-\dim\Aut(\FF_2)+\dim\Aut(\PP^5)=69.$$
\vni

\vni
In conclusion, $\Hh_{16,20,5}$ is irreducible whose general element is a curve of type $(5,6)$ lying on $\FF_0\subset\PP^5$. 

For the fibre of the moduli map, we have $\mu^{-1}\mu(C)=\Aut{(\PP^5)}C$ for every smooth $C\subset\FF_i, i=0,2.$ Note that $C$ has a unique $g^1_5$ and is pentagonal by the Castelnuovo-Severi inequality. For $C\in |\Oo_Q(5,6)|$, $C$ has a unique $g^1_6$
by observing that $6$-points fails to impose independent conditions on $|K_Q+C|=|(3,4)|$ only if $6$-points are collinear and hence is cut out by the second ruling of $Q$, which in turn implies that  a very ample $g^5_{16}=\phi_{|C} (|g^1_5+g^1_6|)$ 
is unique, where $\phi: Q\to \FF_0\subset\PP^5$ is the isomorphism induced by $|\Oo_Q(1,2)|$.

For $C\in|5C_0+11f|$ on $\FF_2$, one easily computes that a very ample $|aC_0+bf|$ on $\FF_2$ with $\dim|aC_0+bf|=5$ such that $$(aC_0+bf)\cdot (5C_0+11f)=a+5b=16=\deg C$$ and satisfying \eqref{cone} is $|C_0+3f|$. Thus a very ample $g^5_{16}$ is unique cut out by hyperplanes in $\PP^5$. Thus we again have 
$\mu^{-1}\mu(C)=\Aut(\PP^5)C$ for any smooth $C\subset\FF_2$. We have
\begin{align*}\dim\mu(\Hh_{16,20,5})&=\dim\Hh_{16,20,5}-\dim\mu^{-1}\mu(C)\\\nonumber&=\dim\Hh_{16,20,5}-\dim\Aut(\PP^5)=35.
\end{align*}
and therefore $\mu(\Hh_{16,20,5})\subsetneq\Mm^1_{g,5}\setminus\Mm^1_{g,4}$ is in 
a sub-locus of $\Mm_g$ with big codimension $3g-3-\dim\mu(\Hh_{16,20,5})=22=g+2$.
\vni

\item[(iii)] For $g=19$, there does not exist a smooth curve in $\PP^5$ with $(d,g)=(16,19)$ since there is no common integer solutions for  \eqref{lsdimension} or  \eqref{cone} with $g=p_a(C)=19$.
\end{itemize}
\end{proof}

For the next case $g=18$, there appear more components of $\Hh_{16,18,5}$ depending 
on surfaces $S\subset \PP^5$ in which a general $C\in\Hh_{16,18,5}$ may sit.
\begin{thm}\label{g=18} $\Hh_{16,18,5}=\HL{16,18,5}$ is reducible with three components $\Hh_i ~ (i=1,2,3)$ such that 
\begin{itemize} 
\item[(i)] $\Hh_1$ dominates $\Mm^1_{g,3}$, a general $C\in\Hh_1$ lies on a rational normal scroll, $\mu^{-1}\mu(C)=\Aut(\PP^5)C$ and $\dim\Hh_1=72$.
\item[(ii)] A general $C\in\Hh_2$ is hexagonal, lies on a smooth del Pezzo surface in $\PP^5$  and 
$\dim\Hh_2=68$.
\item[(ii)] A general $C\in\Hh_3$ is a triple covering of an elliptic curve which lies lies on an elliptic cone and  
$\dim\Hh_3=64$.
\end{itemize}
\end{thm}
\begin{proof} Since $\pi(16,6)=15<g=18$, we recall $\Hh_{16,18,5}=\HL{16,18,5}$ as we remarked in the proof of the previous theorem.
Since $\pi_1(16,5)=18$, $C$ lies on a surface $S\subset\PP^5$ with $\deg S\le 5$. 

\begin{enumerate}
\item[(i)] We assume $\deg S=4$.
\begin{itemize}
\item[(a)] There does not exist a smooth plane curve of genus $g=18$, hence
$S$ cannot be a Veronese surface.
\item[(b)] If $S$ is a rational normal surface scroll, by solving \eqref{scrolldegree} with $p_a(C)=18$, we get
$a=3, b=4$ and $C\in|3H+4L|$ is trigonal with a unique trigonal pencil by the Castelnuovo-Severi inequality. By adjunction formula, we have
\begin{align*}
|K_C(-1)|&=|K_S+C-H|_{|C}=|(-2H+2L)+(3H+4L)-H|_{|C}\\&=|6L|_{|C}.\nonumber
\end{align*}
 Since the ruling $|L|$ cut out a (unique) $g^1_3$ on $C$ we have $$|\Oo_C(1)|=|K_C(-6g^1_3)|.$$

\vni
Conversely, on a general trigonal curve, $|K_C(-6g^1_3)|$ is very ample by Lemma \ref{kveryample11}. Thus, over the irreducible locus $\Mm^1_{g,3}$ of trigonal curves, there is an irreducible family $\Gg\subset\Gg^5_{16}$
consisting of very ample linear series $g^5_{16}=|K_C(-6g^1_3)|$ and hence we have an irreducible family $\Hh_1$ which is an $\Aut(\PP^5)$-bundle over the irreducible family $\Gg$.
Thus we have 
\begin{align*}
\dim\Hh_1&=\dim\Mm^1_{g,3}+\dim\Aut(\PP^5)=2g+1+35=72\\&>\Xx(16,18,5)=62.
\end{align*}
Since $g^1_3$ on $C\in\Hh_1$ is unique, $|\Oo_C(1)|=|K_C(-6g^1_3)|$ is unique hence 
$\mu^{-1}\mu(C)=\Aut(\PP^5)C$.

\item[(c)]
There does not exist a smooth curve in $\PP^5$ with $(d,g)=(16,18)$ lying on a rational normal cone since there is no integer solution for  the equation \eqref{cone} with $g=p_a(C)=18$.
\end{itemize}
\item[(ii)] 
Let $S\subset\PP^5$ be a quintic surface. By the classification of quintic surfaces in $\PP^5$, 
$S$ is one of the following;  
\begin{itemize}

\item[(ii-1)] a del Pezzo surface possibly with finitely many isolated double points

\item[(ii-2)] a cone over a smooth quintic elliptic curve in $\PP^4$

\item[(ii-3)] a cone over a rational quintic curve (either smooth or singular) in $\PP^4$ 

\item[(ii-4)] an image of a projection into $\PP^5$ of a surface  $\w{S}\subset\PP^6$ of minimal degree $5$ with center of projection $p\notin\w{S}$.

\end{itemize}
We note that the  last case (ii-4) is not possible; $C$ is linearly normal since  $\pi(16,6)=15$. For the case (ii-3), we have either $\dim\mathrm{Sing}(S)=0$ or $S$ has a double line. In both cases, $S$ is 
the image of a linear projection of a cone $\w{S}\subset\PP^6$  over a rational normal curve $\w{R}\subset\w{H}\cong\PP^5$ with center of projection $p\in\w{H}\setminus\w{R}$. 
This is not 
possible either since $C\subset S$ is linearly normal.

\vni
\item[(ii-1-1)] Assume $S$ is a smooth del Pezzo surface, which is the image of $\PP^2_4$ - the projective plane blown up at $4$-general points - embedded into $\PP^5$ by the anti-canonical system $|-K_S|$.
\vni

Before proceeding, we recall some standard notation concerning linear systems and divisors on a blown up projective plane. Let $\PP^2_s$ be the rational surface $\PP^2$ blown up at $s$ general points. Let $e_i$ be the class of the exceptional divisor
$E_i$ and $l$ be the class of a line $L$ in $\PP^2$. For integers  $b_1\ge b_2\ge\cdots\ge b_s$, let $(a;b_1,\cdots, b_i, \cdots,b_s)$ denote class of the linear system $|aL-\sum b_i E_i|$ on $\PP^2_s$.  By abuse of notation we use the  expression $(a;b_1,\cdots, b_i, \cdots,b_s)$ for the divisor $aL-\sum b_i E_i$ and \newline$|(a;b_1,\cdots, b_i, \cdots,b_s)|$ for the linear system $|aL-\sum b_i E_i|$. We use the convention 
\vspace{-3pt}
$$(a;b_1^{s_1},\cdots,b_j^{s_j},\cdots,b_t^{s_t}), ~ \sum s_j=s$$ 
when  $b_j$ appears $s_j$ times consecutively  in the linear system $|aL-\sum b_i E_i|$.

\vni

A smooth quintic del Pezzo surface $S\subset\PP^5$ is the image of $\PP^2_4 \lhook\joinrel\xrightarrow{|(3;1^4)|} S\subset\PP^5$. Assume $(a;b_1,\cdots,b_4)$ is very ample. For $C\in (a;b_1,\cdots,b_4)$, we have
\begin{align}\label{delPC11}\deg C &=3a-\sum b_i\nonumber\\ C^2=a^2-\sum b_i^2&=2g-2-K_S\cdot C=2g-2+\deg C.
\end{align}
By Schwartz's inequality, 
\begin{equation}\label{sch11}(\sum b_i)^2\le 4(\sum b_i^2).
\end{equation} 
In our case $\deg C=16, g=18$, by \eqref{delPC11} we have $$16=3a-\sum b_i, a^2-\sum b_i^2=50.$$ Using \eqref{sch11}, we
get $C\in(9;3^3,2)$,  which is  very ample by \cite{sandra}. From this numerical description of $C$, $C$ has a plane model of degree $9$ with three ordinary triple points and a node. Furthermore  $C$ has three base-point-free $g^1_6$'s cut out by lines through triple points. $C$ being hexagonal does not follow directly from the Castelnuovo-Severi inequality. Instead,  one may use \cite[Theorem 2]{Sakai} to verify that $C$ has no $g^1_5$.  We have
$$\dim|(9;3^3,2)|=\tbinom{9+2}{2}-1-3\cdot\tbinom{3+1}{2}-\tbinom{2+1}{2}=33.$$
On the other hand, smooth del Pezzo surfaces in $\PP^5$ form an irreducible family $\Dd(5)$ 
with  $$\dim\Dd(5)=\dim\Aut(\PP^5)-\dim\Aut(\PP^2)+\dim\PP^2\times 4=35.$$ 
Thus we have an irreducible family $\Hh_2\subset\Hh_{16,18,5}$ consisting of smooth curves lying on smooth del Pezzo surfaces with 
$$\dim\Hh_2=\dim|(9;3^3,2)|+\dim\Dd(5)=68>\Xx(16,18,5)=62.$$

The irreducible locus $\Hh_2$ does not dominate $\Mm^1_{g,6}$; a general member in $\Mm^1_{g,6}$ has a unique $g^1_6$, whereas a general member in $\Hh_2$ has at
least three $g^1_6$'s. A general $C\in\Hh_2$ has a base-point-free $g^1_7$ cut out by lines through a double point, hence $C$ is not a triple covering of an elliptic curve by Castelnuovo-Severi inequality.

\vni
\item[(ii-1-2)] Assume that $S$ is a singular del Pezzo with finitely many double points.
We note that  every $C\in \Hh_{15,18,5}$ contained in a singular del Pezzo surface are limits of curves lying on a smooth del Pezzo surface by Proposition \ref{b1}. Thus the family of smooth curves with $(d,g)=(16,18)$ is in the boundary of $\Hh_2$ and does not constitute a full component.
\vni
\item[(ii-2)] Suppose $S$ is a cone over an elliptic curve in $\PP^4\subset\PP^5$.
Recall that a cone $S\subset\PP^r$ over an elliptic curve $E\subset H\cong\PP^{r-1}$ with vertex outside $H$ is the image of the birational morphism $\mathbb{E}_{r}:=\mathbb{P}(\h{O}_{E}\oplus\h{O}_{E}(r))\rightarrow S\subset\PP^r$ induced by the tautological bundle $|\bar{h}|:=|\Oo_{\mathbb{E}_{r}}(1)|=|C_0+rf|$, where $C_0^2=-r$ and $f$ is the fibre of $\mathbb{E}_{r}\stackrel{\eta}{\rightarrow} E$. 
Let $C\subset S$ be an integral curve of degree $d$ with  the strict transformation  $\w{C}$ under  $\mathbb{E}_{r}\rightarrow S$. 
~Setting $k=\w{C}\cdot f$, we have $\w{C}\equiv kC_0+df$, $\deg\eta_{|\w{C}}= \w{C}\cdot f=k$ and
\vspace{-8pt}
\begin{equation}\label{conegenus3}
p_a(\w{C})=
(k-1)(d-\frac{kr}{2})+1,~~
0\le \w{C}\cdot C_0=d-rk=m
\end{equation} 
where $m$ is the multiplicity of $C$ at the vertex.

In our current situation, we have $d=16, k=3, r=5, \, p_a(C)=g=18, ~m=1$ by \eqref{conegenus3}. 
Thus our smooth $C$ passes through the vertex of the cone $S$. Furthermore, $C$ is a triple covering of an elliptic curve induced by $\eta_{|\w{C}}$. 
Thus we may have a further family of smooth curves with $(d,g)=(16,18)$ lying on an elliptic  cone $S$ which
forms an irreducible family $\Hh_3$ with
\begin{align}\label{H31}
\dim\Hh_3&=\dim\PP^5+\dim\Hh_{5,1,4}+\dim|3C_0+16f|\nonumber\\&=5+\Xx(5,1,4)+\dim|3C_0+16f|\nonumber\\&=5+25+34=64>\Xx(16,18,5)=62
\end{align}
where the first summand is the degree of freedom of the choice of the vertex of the elliptic cone $S$ in $\PP^5$, the second summand is the number of parameters of elliptic quintic curves $E$ in a complementary  hyperplane $H\cong\PP^4$ -- which is irreducible by a result of Ein \cite{E2} with 
$\dim\Hh_{5,1,4}=\Xx(5,1,4)=25$ -- and the last summand is $\dim|\w{C}|$ on
$\EE_5$ which follows from K\"unnuth formula. The inequality \eqref{H31} suggests that  it may a priori be possible that $\Hh_3$ may form an irreducible component
of $\Hh_{16,18,5}$ different from $\Hh_2$. 

It is known that any smooth limit of $k_1$-fold covers of genus $h_1$ curves must be a $k_2$-fold cover of a genus $h_2$ curve for some $k_2 \le k_1$ and $h_2 \le h_1$.  This basically follows from the theory of admissible covers; when branched covers degenerate, one can keep the branch points separated and allow the domain/target curves to become nodal. In this process the (arithmetic) genus remains the same, but the domain/target curves may become reducible whereas each irreducible component may have smaller genus and the degree of the cover restricted to a component may possibly become smaller; cf. [15, 3G] and [16].
Therefore $\Hh_3$ is not in the boundary of $\Hh_2$ and therefore is an irreducible component of $\Hh_{16,18,5}$. 
\end{enumerate}
\end{proof}

\begin{rmk} 
\begin{itemize}
\item[(i)]
We  recall that the component $\Hh_1\subset\Hh_{16,18,5}$ dominates $\Mm^1_{g,3}$ and  hence $$\codim_{\Mm_g}\mu(\Hh_1)=
\dim\Mm_g-\dim\Mm^1_{g,3}=g-4.$$ This example is related to the following conjecture (or a question) raised in
\cite[p. 142]{h1}: 
\vni
{\it If $\h{H}$ is any component of the Hilbert scheme $\Hh_{d,g,r}$ such that 
the image of the rational map $\h{H}\longrightarrow \h{M}_g$ has codimension $g-4$ or less, then
$$\dim\h{H}=\h{X}{(d,g.r)}.$$} 

\item[(ii)]
Obviously the component $\Hh_1$ in Theorem \ref{g=18} provides a negative answer to the above.
There are other examples of similar kind, e.g. $\Hh_{15,16,5}$, $\Hh_{15,13,5}$ and some others, having a component of dimension {\it larger than expected} and dominating $\Mm^1_{g,3}$ as well; cf. \cite[Theorem 5.1,Theorem 3.2]{edinburgh}. There is a notable difference between $\Hh_{15,16,5}$  and  $\Hh_{15,13,5}$. The larger than expected component of $\Hh_{15,16,5}$ is a component of $\HL{15,16,5}$. On the other hand, the larger than expected component of $\Hh_{15,13,5}$ is 
not a component of $\HL{15,13,5}$. Such a component entirely consisting of non-linearly normal curves has been found in the literature; cf. \cite[Proposition 3.5]{CKP}. As far as the author knows, all the larger than expected components $\Hh\subset\HL{d,g,r}$ with $\codim\mu({\Hh})= g-4$ dominate the locus  $\Mm^1_{g,3}$ of trigonal curves. 

\item[(iii)]
It would be interesting to know if there are further examples of this kind, i.e. a component $\Hh\subset \Hh_{d,g,r}$ such that $\dim\Hh>\Xx(d,g,r)$ and $\codim_{\Mm_g}\mu(\Hh)= g-4$ which {\it does not} dominate
$\Mm^1_{g,3}$. 
\item[(iv)] Taking these known examples which we listed above into our consideration,  
one may want to modify the conjecture (i) in the following compromised and tentative form:

\vni
{\it If $\h{H}$ is any component of the Hilbert scheme $\HL{d,g,r}$ of \textit{\textbf{linearly normal}}
curves such that 
the image of the rational map $\h{H}\longrightarrow \h{M}_g$ has codimension \textit{\textbf{strictly less than}} $g-4$, then
$$\dim\h{H}=\h{X}{(d,g.r)}.$$} 
\item[(v)] However, there are several  examples of components $\Hh\subset\HL{d,g,r}$ violating the tentative conjecture (iv) which dominates
$\Mm^1_{g,k}$ for some $k\ge 4$ with $\dim\Hh>\Xx(d,g,r)$ hence $$\codim\mu{(\Hh)}=3g-3-\dim\Mm^1_{g,k}= g+2-2k\le g-6.$$  To name a typical such one, one may check that $\HL{d,g,r}$ with $d=g+r-3$, $r+7\le g\le 2r+1$ has a component $\Hh$ such that $\dim\Hh>\Xx(d,g,r)$ dominating $\Mm^1_{g,k}$ for some $k\ge 4$, hence $\codim\mu{(\Hh)}=g+2-2k$.  However, this example seems to be rather crude and artificial. For details, the readers are advised to consult the construction discussed in the proof of
\cite[Theorem 3.9]{lengthy}.
\end{itemize}
\end{rmk}

\vni
\begin{thm} $\Hh_{16,17,5}$ is reducible with at least two components.
\end{thm}
\begin{proof} In previous cases, we had $g\ge\pi_1(16,5)=18$ and hence $C\in\Hh_{16,17,5}$ lies on a surface $S$ with $\deg S\le 5$, which has a good and precise classification. By analyzing curves on such surfaces of low degrees, it was possible to identify all the irreducible components of 
$\Hh_{16,g,5}$ for $g\ge 18$. 
However, in the current case, $g=17<\pi_1(16,5)=18$ and therefore there is no guarantee that $C$ lies on a surface $S\subset\PP^5$ with low $\deg S$ . Thus, 
we may not be able to locate all the irreducible components as we did for $g\ge 18$. Instead, we will only show the reducibility of $\Hh_{16,17,5}$ just by considering some plausible and sporadic situation. 

Take a smooth $C\in\Hh_{16,17,5}$. From the long cohomology sequence of the standard exact sequence
$$0\to \Ii_C(2)\to \Oo_{\PP^5}(2)\to\Oo_C(2)\to 0$$
and by Riemann-Roch, we have $$16\le h^0(C, \Oo_C(2))=32-g+1+h^1(C,\Oo_C(2))\le 17$$ and hence $$h^0(\PP^5,\Ii_C(2))\ge h^0(\PP^5,\Oo(2))-h^0(C,\Oo_C(2))\ge 4.$$ Motivated by this, we let 
$C$ be a smooth complete intersection of $4$ independent quadric hypersurfaces in $\PP^5$.  In general, for a complete intersection $C$ of $n-1$ hypersurfaces of degrees $a_1,\cdots a_{n-1}$ in $\PP^n$, $$p_a(C)=1+\frac{1}{2}a_1\cdots a_{n-1}(\sum a_i-n-1)$$ by successive application of adjunction. Thus, our 
$C\subset\PP^5$ has genus $g=1+2^3(4\cdot 2-5-1)=17$. Since $C$ is a complete intersection, $C$ is AcM, $h^1(\PP^5,\Ii_C(2))=0$, $h^0(C,\Oo_C(2))=17$ and $|\Oo_C(1)|$ is semi-canonical. We count the dimension of the family $\Hh_1$ consisting of complete intersections of $4$ independent quadrics. Obviously, 
\begin{align}\dim\Hh_1&=\dim\GG(3,\PP(H^0(\PP^5,\Oo(2))))=\dim\GG(3,20)\\&=4\cdot 17>\Xx(16,17,5)=64.
\end{align}
To see that the family $\Hh_1$ is indeed dense in a component of $\Hh_{16, 17,5}$,  we note that $N_{C,\PP^5}=\Oo_C(2)^{\oplus 4}$ and hence $$T_C\Hh_1=h^0(C,N_{C,\PP^5})=4\cdot h^0(C,\Oo_C(2))=68=\dim\Hh_1.$$
Thus the closure $\w{\Hh}_1$ is a component of $\Hh_{16,17,5}$ which is $\w{\Hh}_1$ is generically reduced and superabundant. 

\vni

We now consider a surface $S$ containing $C$ with $\deg S$ smaller.
We assume $C\subset S\subset\PP^5$ and $\deg S=4$; 

\vni
(i) $S$ cannot be a Veronese surface since $\binom{e-1}{2}\neq g=17$ for any $e\in\NN$. 

\vni
(ii) $S$ cannot be a rational normal surface scroll since there is no integer solution for \eqref{scrolldegree} with $(d,g)=(16,17)$. 

\vni
(iii) $S$ cannot be a rational normal cone by \eqref{cone}.

\vni

We assume $\deg S=5$ and $S$ is a smooth del Pezzo in $\PP^5$. By \eqref{delPC11} and 
\eqref{sch11} with $\deg C=16, \, p_a(C)=g=17$, one get $8\le a\le 11$ . If $a=8$,
we have $C\in(8;2^4)$ or $C\in(10;4^3,2)$ if $a=10$. Since general one in $(8;2^4)$ is isomorphic to the one in $(10;4^3,2)$ via a quadratic transformation, we only consider $C\in(8;2^4)$. By an usual dimension count,
$$\dim(8;2^4)=\dim\PP(H^0(\PP^2,\Oo(8)))-4\cdot\tbinom{2+1}{2}=\tbinom{8+2}{2}-1-4\cdot\tbinom{2+1}{2}=32,$$ thus there is an irreducible family $\Hh_2$ consisting of smooth curves
lying on smooth del Pezzo surfaces with 
$$\dim\Hh_2=\dim(8;2^4)+\dim\Aut(\PP^5)=67>\Xx(16,17,5).$$
We consider the exact sequence
\begin{align}
0\to\Ii_{C,S}(2)&=\Oo_S(-C)\otimes\Oo_S(2)=\Oo_S((-(8; 2^4)+2(3;1^4)))\nonumber\\&=\Oo_S((-2;0^4))\to\Oo_S(2)=\Oo_S(2(3;1^4))\to\Oo_C(2)\to 0.
\end{align}
Since $\Oo_S((3;1^4))$ is 
very ample, $h^1(S, \Oo_S(2))=h^1(S,\Oo_S(-3(3;1^4)))=0$ by duality and Kodaira vanishing theorem. By duality again, 
\begin{align}
h^2(S,\Ii_{C,S}(2))&=h^2(S, \Oo_S((-2;0^4)))=h^0(S, \Oo_S(-(-2; 0^4)+K_S)))\nonumber\\&=
h^0(S, \Oo_S((-1; -1^4))=0.
\end{align} Hence $h^1(C, \Oo_C(2))=0$ and  $|\Oo_C(1)|$ is not semi-canonical. 
By upper semicontinuity of cohomology, $\Hh_1$ does not contain $\Hh_2$ in its boundary. Therefore we identified at least two components. 

We may continue to consider further surfaces $S\subset\PP^5$ with $\deg S\ge 6$ which may contain $C$. For example, a Bordiga surface in $\PP^5$ which is the image of $\PP^2_9 \lhook\joinrel\xrightarrow{|(4;1^9)|} S\subset\PP^5$ or the surface $\PP^2_7\lhook\joinrel\xrightarrow{|(4;2,1^6)|} S\subset\PP^5$. However it turns out that 
computation gets rather complicated and it seems that there is not much hope 
to come up with a complete list of components with our imperfect techniques. 
\end{proof}

In the next theorem, we treat curves with $g=14$. We skip $g=15, 16$ cases because
it is more involved than previous cases and we have not been able to reach a meaningful
conclusion. 

\begin{thm}
$\Hh_{16,14,5}=\HL{16,14,5}$ is irreducible with $\dim\Hh_{16,14,5}=\Xx(16,14,5)$.
For a general $C\in\Hh_{16,14,5}$, $\mathrm{gon}(C)=8$.
\end{thm}
\begin{proof}
The irreducibility of $\HL{16,14,5}$ was shown in \cite[Theorem 2.5]{JPAA} under the condition
$\HL{16,14,5}\neq\emptyset$.  The non-emptiness of $\HL{16,14,5}$ was shown in \cite[Theorem 3.7]{lengthy} in a more general context, where the author used in particular the existence of very-ample linear series $|K-2g^1_5|$ on a general $5$-gonal curve, which does not follow directly from 
Lemma \ref{kveryample11}. Instead, this follows from Lemma \ref{kveryample}. On the other hand, the irreducible family of smooth curves with $(d,g)=(16,14)$ arising this way is not a component of $\HL{16,14,5}$ since $\dim\Mm^1_{g,5}=33<\lambda(16,14,5)=35$. Hence for a general $C\in\HL{16,14,5}$, $\mathrm{gon}(C)\ge 6$. 
For an estimate the upper bound of the gonality of a general element  $C\in\HL{16,14,5}$ and computing $\dim\HL{16,14,5}$, we consider 
$\Ee=|K_C(-1)|=g^2_{10}$. After eliminating the possibility for $\Ee$ being compounded, one 
may deduce that $\Ee$ is birationally very ample with empty base locus. Hence $\Ee$ induces a singular plane model of degree $10$. Such a curve is in the equi-generic
Severi variety $\Sigma_{10,g}$ consisting of curves of degree $10$ and $g=14$. Recall
that a general member of $\Sigma_{10,g}$ is a nodal curve with exactly $\binom{10-1}{2}-g=22$ nodes as its only singularities. Thus we have a natural rational map
\begin{align*}\label{none}
\HL{16,14,5}/\Aut(\PP^5)&\stackrel{\Psi}{\dasharrow} \Sigma_{10,g}/\Aut(\PP^2)\nonumber\\
|\Oo_C(1)|&\mapsto|K_C(-1)|
\end{align*}
 We now show that $\Psi$ is dominant which would imply that 
 $\Psi$ is  generically one-to-one and as a consequence
 \begin{align*}\dim\HL{16,14,5}&-\dim\Aut(\PP^5)=\dim\Sigma_{10,g}-\dim\Aut(\PP^2)\nonumber\\&=\Xx({10,g,2})-\dim\Aut(\PP^2)=\lambda(10,g,2)=\lambda(16,g,5).
 \end{align*}
 
 Choose a general element in $\Sigma_{10,g}/\Aut(\PP^2)$, i.e.
 a general element of $\Gg\subset\Gg^2_{10}$, where $\Gg$ is the irreducible locus 
 consisting of birationally very ample $g^2_{10}$'s inducing a nodal plane curve. $\Gg$ is known to be irreducible by a theorem of Harris \cite{H2}. 
 
 We first fix two integers $e=\deg\Ee=10$ and 
 \begin{equation*}\label{delta}\delta=\binom{e-1}{2}-g=36-g=22>0.
 \end{equation*}
 Let $$\Delta =\{p_1, \cdots , p_{\delta}\}\in{\text Sym}^\delta(\PP^2)$$ be a general set of $\delta$ points in $\PP^2$. 
 Note that $\Sigma_{e, g}\neq\emptyset$; there exists a curve $E\in\Sigma_{e,g}$ having nodes at $\Delta$ and no further singularities. With the same choice of a general $\Delta$, there also exists a curve $E'\in\Sigma_{e-1, h}$ of genus $h=\binom{e-2}{2}-\delta=6$ and degree $e-1$ having nodes at $\Delta$.

Let $S:=\PP^2_\delta\stackrel{\pi}{\rightarrow}\PP^2$ be the blowing up of $\PP^2$ at  $\Delta$. 
We consider the linear system 
$$|H|=|\pi^*((e-4)L)-\sum_{i=1}^\delta E_i| $$ on $S$.
Recall that by a result of  d'Almeida and Hirschowitz \cite[Theorem 0]{Coppens}, the linear system $|\pi^*(tL)-\sum_{i=1}^\delta E_i|$ on $S$ is very ample if 
\begin{equation}\label{almeida}\delta\le \frac{(t+3)t}{2}-5.
\end{equation} The inequality (\ref{almeida}) holds for $t=6$ whence $|H|$ is very ample.
The linear system $$|C|=|\pi^*(eL)-\sum_{i=1}^\delta 2E_i|$$
contains the non-singular model $C$
of the nodal plane curve $E$ and $$H\cdot C=((e-4)l-\sum_{i=1}^\delta e_i)\cdot (el-\sum_{i=1}^\delta 2e_i)=e(e-4)-2\delta=16.$$
Recall that $\Delta =\{p_1,\cdots ,p_\delta\}\in{\text Sym}^\delta(\PP^2)$ is general and there is a nodal curve $E'\in \Sigma_{e-1, h} $ with nodes at $\Delta$. Therefore we may assume that $\Delta$ imposes independent conditions on the linear system of curves of any degree $m\ge e-4$; cf. \cite[Exercise 11, p.54]{ACGH} or  \cite[(1.51, p.31)]{H3}.
\vni
Hence by the projection formula, we obtain 
\begin{align*}
h^0(S,\h{O}_S(H))&=h^0(\PP^2, \pi_*(\h{O}_S(H)))=h^0(\PP^2,\pi_*(\pi^*((e-4)L)-\sum_{i=1}^\delta E_i))\\
&=h^0(\PP^2,\h{O}_{\PP^2}((e-4)L-\Delta))=
1+\frac{1}{2}(e-4)(e-1)-\delta\\&=6.
\end{align*}
Thus, we may deduce that the non-singular model $C\subset \PP^2_\delta$ of a nodal plane curve $E$ of degree $e=10$ in the linear system $|C|=|\pi^*(eL)-\sum_{i=1}^\delta 2E_i|$ is embedded into $\PP^5$ as a curve of degree $d=16$ by the very ample linear 
system  $|H|=|\pi^*((e-4)L)-\sum_{i=1}^\delta E_i)|.$

Furthermore,  by our choice of general $\delta$ points $\Delta\subset\PP^2$, which imposes independent conditions on curves of any degree $m\ge e-4$, the linear system $g^2_e=|\pi^*(L)_{|C}|$ cut out on $C$ by lines in $\PP^2$ is complete \cite[Exercise 24, p.57]{ACGH}. Hence it follows that the residual series of $|\pi^*(L)_{|C}|$, which is 
\begin{align*}
|(K_S+C&-\pi^*(L))_{|C}|\\&=|(-\pi^*(3L)+\sum E_i+\pi^*(eL)-2\sum E_i-\pi^*(L))_{|C}|\\
&=|(\pi^*((e-4)L)-\sum E_i)_{|C}|=|H_{|C}|
\end{align*}
is  a complete very ample $g^5_{16}$ on $C$. Hence $\Psi$ is dominant.

On the plane model of $C$, there is a base-point-free $g^1_8$ cut out by 
lines through one of the nodes, hence $\mathrm{gon}(C)\le 8$. To see that $\mathrm{gon}(C)=8$, 
we use \cite[Theorem 1]{Sakai} instead of going through complicated details for the non-existence a base-point-free $g^1_7$.

We finally show that the irreducible $\HL{16,14,5}$ is the only component of $\Hh_{16,14,5}$.
We assume the existence of  $\Hh\subset\Hh_{d,g,r}$, a component whose general element is a not linearly normal. Choose general $C\in\Hh$. By Castelnuovo genus bound, $\pi(16,6)=15$, $\pi(16,7)=12$, thus we may assume that $|\Oo_C(1)|=g^6_{16}$, $\Ee=|K_C(-1)|=g^3_{10}$. We may exclude $\Ee$ being  compounded by Lemma \ref{easylemma1}.

Thus $\Ee$ is birationally very ample and
$\mathrm{Bs}(\Ee)=\Delta=\emptyset$ since $\pi(9,3)=12$. Let $\ce\subset\PP^3$ be the curve induced by $\Ee$. $\ce$ does not lie on an irreducible cubic surface; $g\le \pi_1(d,3)\le\frac{1}{6}(d-1)(d-2)=12$, a contradiction.
Therefore $\ce$ lies on a quadric surface, either a smooth quadric or a quadric cone.  Possible 
types $(a,b)$ of curves on a smooth quadric in $\PP^3$ of degree $10$  is $(a,b)\in\{(5,5), (4,6), (3,7), (2,8)\}$. 
The last two cases are not possible by genus reason. 

(i) If $(a,b)=(5,5)$, $p_a(\ce)=16$ and $\ce$ is singular. Assume that $\ce$ lies on 
smooth quadric and we also assume that $\ce$ is nodal with $\delta=p_a(\ce)-g=2$. The Severi variety $\Sigma_{\FF_0,\Ll, \delta}$ of nodal curves on $Q=\FF_0$ in the linear system $\Ll=|(5,5)|$ with $\delta$ nodes
is irreducible by \cite{tyomkin} and
$$\dim \Sigma_{\FF_0,\Ll, \delta}=\dim|(5,5)|-\delta=33.$$
Hence we have the family $\Gg\subset\Gg^3_{10}$ consisting of $g^3_{10}$'s arising this way  -- inducing a morphism from $C$ into $Q$ with the image which is a nodal curve of type $(5,5)$) -- and  
$$\dim\Gg=\dim \Sigma_{\FF_0,\Ll, \delta}-\dim\Aut(\FF_0)=27.$$
Thus we have
\begin{align*}\dim\Hh&=\dim\Gg^\vee+\dim\GG(5,6)+\dim\Aut(\PP^5)\nonumber\\&=\dim\Gg+\dim\GG(5,6)+\dim\Aut(\PP^5)=68<\Xx(16,14,5)=70 
\end{align*}
and the family $\Hh$ does not constitute a full component.\footnote{\,Before carrying out the dimension count, it was necessary for us do check that the residual series
of $g^3_{10}$ inducing a curve in $Q\subset\PP^3$ of type $(5,5)$ with two nodes is very ample.
Indeed this is true and we leave this simple checking to readers. }

(ii) If $\ce\in|(4,6)|$,  $p_a(\ce)=15$ and $\ce$ has a node or a cusp. One may check easily that 
the residual series of $\Ee$ is not very ample. 

We recall that integral curves on a quadric cone in $\PP^3$ are specializations of curves on 
smooth quadrics, therefore we may conclude that $\Hh_{d,g,r}$ is irreducible with the only component $\HL{16,14,15}$. 
\end{proof}

\begin{prop}\label{g=13}
\begin{itemize}
\item[(i)] $\Hh_{16,13,5}=\HL{16,13,5}$ is irreducible of the expected dimension $\Xx(16,13,5)$. 
For a general $C\in\Hh_{16,13,5}$, $\mathrm{gon}(C)=8$ and $\mu^{-1}\mu(C)\cong W^1_8(C)\times\Aut(\PP^5)$.
\item[(ii)] $\Hh_{16,12,5}=\HL{16,12,5}$ is irreducible. For general $C\in\Hh_{16,12,5}$, 
$\mu^{-1}\mu(C)\cong C_6\setminus\Gamma\times\Aut(\PP^5)$, where $\Gamma\subset C_6$
is an irreducible locus with $\dim\Gamma =3$.
\end{itemize} 
\end{prop}
\begin{proof} 
\begin{itemize}
\item[(i)] The irreducibility of 
$\HL{16,13,5}$ under the assumption $\HL{16,13,5}\neq\emptyset$ was shown in \cite[Theorem 2.3]{JPAA}. The non-emptiness of $\HL{16,13,5}$ was  implicitly shown in \cite{lengthy}, which we explain in detail as follows. 

In general,
in the range $d\le g+r$ inside the Brill-Noether range $\rho (d,g,r)\ge 0$, the principal component $\mathcal{H}_0\subset\Hh_{d,g,r}$ -- by definition the unique component dominating $\Mm_g$ -- which has the expected dimension $\Xx (d,g,r)$ is one of the components of $\HL{d,g,r}$ (cf. \cite[2.1 page 70]{H1})  and  hence 
 $\mathcal{H}_{d,g,r}^\mathcal{L}$ is {\it always non-empty} in the Brill-Noether range $\rho(d,g,r)\ge 0$ with $g-d+r\ge 0$.


The only remaining issue is that if the whole $\Hh_{16,13,5}$ is irreducible.  Let $\Hh$ be a component consisting of non-linearly normal curves. 
Let $C\in\Hh$ be a general element. We set $\beta:=\dim\Oo_C(1)\ge 6$ and
$\Ee=|K_C(-1)|=g^{\beta-4}_8$. Note that $6\le \beta\le 7$ by Clifford's theorem and $C$ being non-hyperelliptic. If $\beta=7$ and  $\Ee$ is birationally very ample, we have  $\pi(8, 3)=9$. Thus $\Ee$ is compounded and $C$ is a double cover of an elliptic curve. This is impossible by Lemma \ref{easylemma1}. Thus $\beta=6$. 
\item[(i-A)]
Suppose $\Ee=g^2_8$ is compounded and let $\delta=\deg\Delta$, $\Delta=\mathrm{Bs}(\Ee)$. Since $\Ee$ is special and $C$ is non-hyperelliptic,   $0\le\delta\le 2$.
If $\delta=2$ then $C$ is a double cover of an elliptic curve which is impossible by  Lemma \ref{easylemma1}. Thus $\Ee$ is base-point-free. 
Let $\phi$ be the morphism induced by the compounded $\Ee$.
\item[(i-A-1)]
Suppose $\deg\phi=4$.  Then $C$ is $4$-gonal, $\Ee=2g^1_4$,  $\Dd:=|\Oo_C(1)|=|K_C(-2g^1_4)|=g^6_{16}$. By the Lemma \ref{kveryample11}, $\Dd$ is very ample. 
Thus,  we are in the following situation. Over the locus $\Mm^1_{g,4}$, we have a very ample
$|K_C(-2g^1_4)|$. By taking a $5$-dimensional subseries of $|K_C(-2g^1_4)|$, we have $$\dim\Mm^1_{g,4}+\dim\GG(5,6)=2g+9=35<\lambda(16,13,50)=37,$$  
dimensional irreducible family of $g^5_{16}$'s inside $\Gg^5_{16}$. Thus, this family does not constitute a full component. 
\item[(i-A-2)] Suppose $\deg\phi=2$. Then $C$ is a double cover of a plane quartic $E\subset\PP^2$. We may assume that $E$ is {\it smooth}. Otherwise, the linear series $\w{\Ee}=g^2_4$ on $E$ such that $\phi^*(\w{\Ee})=\Ee$
is non-special and Lemma \ref{easylemma1} applies. 
It can be shown easily that 
$$|K_C-\phi^*(g^2_4)|=|K_C-\phi^*(K_E)|=g^6_{16}$$
is very ample, which may well contribute to a component of $\Hh_{16,13,5}$ other than 
$\HL{16,13,5}$. Recall that the Hurwitz space $\Xx_{n,\gamma}\subset\Mm_g$ -- the locus of  smooth curves of genus $g$ which are degree $n$ ramified coverings of curves of genus $\gamma$ -- is of pure dimension $2g+(2n-3)(1-\gamma)-2$; cf. \cite[Theorem 8.23, p. 828]{ACGH2}. In our  case, the irreducible family of very ample $g^5_{16}$'s which is $5$-dimensional subseries of the residual series of $|\phi^*(K_E)|$ on a double covering 
of smooth plane quartics form a family $\Ff$ with 
\begin{align*}\dim\Ff&=\dim\Xx_{2,3}+\dim\GG(5,6)=2g-4+\dim\GG(5,6)\\&=28<\lambda(16,13,5)=37,
\end{align*}
hence $\Ff$ does not constitute a
full component. 
\item[(i-B)]
Suppose $\Ee=g^2_8$ is birationally very ample. The Severi variety of plane curves of 
degree $8$ with $\delta=\binom{8-1}{2}-g=8$ nodes has dimension
$$\dim\Sigma_{8,g}=\tbinom{8+2}{2}-1-\delta=36.$$
Therefore the family $\Ff_{\Sigma_{8,g}}$ consisting of $g^5_{16}$'s arising this way has dimension
$$\dim\Sigma_{8,g}/\Aut(\PP^2)+\dim\GG(5,6)=34<
\lambda(16,13,5), $$ hence  this family $\Ff_{\Sigma_{8,g}}$ does not constitute a full component and this finishes the proof of $\Hh_{16,13,5}=\HL{16,13,5}$. 
\end{itemize}
\noindent
\vni

Since $\rho(16,13,5)\ge 0$, $\Hh_{16,13,5}$ dominates $\Mm_{13}$. By Brill-Noether theorem, $W^1_7(C)=\emptyset$, $W^2_8(C)=\emptyset$ and $\dim W^1_8(C)=1$ for general $C\in\Mm_{13}$.  Thus $|K_C-g^1_8|=g^5_{16}$ is very ample and therefore we have $$\mu^{-1}\mu(C)\cong W^1_8(C)\times\Aut(\PP^5)$$ for a general $C\in\Hh_{16,13,5}$.
\begin{itemize}
\item[(ii)] 
The irreducibility and the non-emptiness of $\HL{16,12,5}$ is shown in \cite[Theorem 2.2]{JPAA}  by successive projection from the 
canonical curve. To show that $\Hh_{16,12,5}=\HL{16,12,5}$, we argue in a similar way as in  
the previous case $\HL{16,13,5}$.
Assume the existence of a component  $\Hh\subset\Hh_{16,12,5}$ consisting of non linearly normal curves. Choose general 
$C\in \Hh$ and set $\beta:=\dim|\Oo_C(1)|\ge 6$, $\Ee=|K_C(-1)|=g^{\beta-5}_{6}.$
We have 
$1\le\beta-5\lneq 3$ by Clifford's theorem and  since $C$ is non-hyperelliptic. 

\item[(ii-A)]
If $\beta-5=2$, $\Ee$ is not birationally very ample since $\pi(6,2)=10<g$. Thus $\Ee$ is compounded,
$C$ is trigonal and $\Ee=2g^1_3$. Note that $|K_C(-2g^1_3)|$ is very ample by Lemma \ref{kveryample11}. The family $\Ff_3\subset\Gg^5_{16}$ consisting of very ample $g^5_{16}$'s which are general subseries 
of $|K_C(-2g^1_3)|$ on trigonal curves has
$$\dim\Ff_3=\dim\Mm^1_{g,3}+\dim\GG(5,7)=37<\lambda(16,12,5)=39,$$
and hence does not contribute to a full component. 
\item[(ii-B)]
If $\beta-5=1$,  $|K_C(-1)|=g^1_6$. We recall a well-known fact  \cite{B} that 
the Clifford index $c$ of a general $k$-gonal curve
is computed by the unique pencil $g^1_k$ and there is no $g^r_d$ with $d\le g-1$ and $d-2r=c$. In particular,  on a general hexagonal curve $C$,  $|K_C-g^1_6|=g^6_{16}$ is very ample, i.e.  $\dim|g^1_6+p+q|=1$ for any $p+q\in C_2$.
However, the family $\Ff_6$ consisting of $g^5_{16}$'s which are subseries of the form $|K_C-g^1_6|$ has $$\dim\Ff_6=\dim\Mm^1_{g,6}+\dim\GG(5,6)=2g+13=37<\lambda(16,12,5)=39,$$
and hence does not constitute a full component. 
\end{itemize} Thus $\Hh_{16,12,5}=\HL{16,12,5}$ is irreducible and dominates $\Mm_{12}$.
A general $C\in \Hh_{16,12,5}$ is $7=[\frac{g+3}{2}]$-gonal with $\dim W^1_7(C)=0$. Set  $C^r_d(C)=\{D\in C_d : r(D)\ge r\}$. By Riemann-Roch, a complete 
$g^5_{16}$ on any curve $C$ of genus $g=12$ is of the form $|K_C-D|$, $D\in C_6\setminus C^1_6(C)$. 
Set $$\Gamma:=\{D\in C_6: D\le F \mathrm{ ~ for ~some ~} F\in C^1_8(C)\}.$$
On a general $C\in \Mm_{12}$, $C^1_6(C)=\emptyset$, $C^1_8(C)\neq\emptyset$ is irreducible with $$\dim C^1_8(C)=\rho(8,12,1)+1=3$$ by Brill-Noether theorem. Hence $\Gamma$ is irreducible with $\dim\Gamma=3$. It finally follows that for every $D\in C_6\setminus\Gamma$, $|K_C-D|$ is complete, very ample and hence
$$\mu^{-1}\mu(C)=C_6\setminus\Gamma\times\Aut(\PP^5)$$
for general $C\in\Hh_{16,12,5}$.

\end{proof}
\begin{rmk} The result we showed in Proposition \ref{g=13} asserts that $\Hh_{g+3,g,5}$ and $\Hh_{g+4,g,5}$  is irreducible for $g=13$ and $g=12$, which is {\it slightly beyond} the range $d\ge g+5$ conjectured by Severi in \cite{Sev}. 
\end{rmk}

For the lower genus $g\le 11$, it is enough to recall the following rather crude bound concerning the 
irreducibility of $\Hh_{d,g,r}$ when $d$ is big enough compared with the genus $g$\,; cf. \cite[Theorem 2.1]{Keem}.
\begin{prop} $d\ge 2g-7, g+r\le d, r\ge 3 $, $\Hh_{d,g,r}$ is irreducible and non-empty
dominating $\Mm_g$\,. 
\end{prop} 
\begin{cor}$\Hh_{16,g,5}$ is irreducible dominating $\Mm_g$ for any $g\le 11$. 
\begin{itemize}
\item[(i)] If $g=11$, $\Hh_{16,g,5}=\HL{16,g,5}$.
\item[(ii)] If $g\le 10$, $\HL{16,g,5}=\emptyset$ and $\Hh_{16,g,5}\neq\emptyset.$
\end{itemize}
\end{cor}

\section{Digression,  remarks on $\Hh_{16,16,5}$ and $\Hh_{16,15,5}$}

We have not been able to determine the irreducibility of $\Hh_{16,16,5}$ (or $\Hh_{16,15,5}$). We even failed to locate any of its components. In this final section with loose end, we would like to list up some irreducible families $\Hh\subset\Hh_{d,g,r}$ with $\dim\Hh\ge\Xx(d,g,r)$, which came out from our insufficient investigation on $\Hh_{16,16,5}$ and $\Hh_{16,15,5}$. 

\subsection{Irreducible families in $\Hh_{16,16,5}$}

We only state intermediate results which may contribute toward a reasonable settlement for the study of $\Hh_{16,16,5}$. We omit details.

\begin{rmk}
It is easy to see that $\Hh_{16,16,5}\neq\emptyset$; cf. (ii) or (iv) below. Let $\Hh\subset\Hh_{16,16,5}$ be a component.
\begin{itemize}
\item[(i)] There is no surface $S\subset\PP^5$ with $\deg S=4$ which contains a general element of $\Hh$. 

\item[(ii)] If $S\subset\PP^5$ is a smooth del Pezzo, a smooth $C\subset S$ with $(d,g)=(16,16)$
is in $|(8;3,2^2,1)|$, whose dual curve $C^\vee\subset \PP^4$ induced by $|K_C(-1)|$ lies on a rational normal scroll $T\subset\PP^4$, $\deg C^\vee=13$, $|K_C(-1)|$ has one base point. $C^\vee$ is singular having two nodes with the strict transformation  $\w{C}^\vee\in|(8;3,2^2)|$ under the blow up
$\PP^2_3\to T\cong\PP^2_1$ at the two nodes. 
Furthermore, 
$C$ is the image of the projection from $\w{C}^\vee\subset\PP^2_3\subset\PP^6$ with center of projection at the base point of $|K_C(-1)|$.

\item[(iii)] The irreducible family of curves $\Hh_1\subset\Hh_{16,16,5}$ lying on smooth del Pezzo surfaces has the expected dimension $\Xx(16,16,5)$.

\item[(iv)] There is an irreducible family $\Hh_2\subset\Hh_{16,16,5}$ consisting of curves which lie on a smooth surface $W\cong\PP^2_7\to\PP^5$ embedded by $|(4;2, 1^6)|$, $\deg W=6$. For a general $C\in\Hh_2$,
$|K_C(-1)|$ is base-point-free,  birationally very ample and $C^\vee\subset\PP^4$ is singular with $6$-nodes sitting on a rational normal scroll, $\w{C}^\vee\in|(9;4,2^6)|$, $\dim\Hh_2=\Xx(16,16,5)+1$.
\item[(v)] However, it has not been possible to draw any meaningful conclusion from the above discussion. We may need to consider surfaces of another kind (probably of higher degrees) which
may contain a general $C\in \Hh$.
\end{itemize}
\end{rmk}

\subsection{Irreducible families in $\Hh_{16,15,5}$}
As the genus $g$ gets lower but not too low, there appears more subtle situation which we could not handle 
with our conventional method we uses so far. 
\begin{rmk}\begin{itemize}
\item[(i)] $\Hh_{16,15,5}\neq\emptyset$, e.g. $C\subset\PP^2_7$, $C\in (9;3^3,2^4)$ embedded 
by $|(4;2,1^6)|$. 

\item[(ii)] There is no surface $S\subset\PP^5$ with $\deg S=4$ which contains a general element of $\Hh$. 

\item[(iii)] The curve $C$ mentioned in (i), $|K_C(-1)|=g^3_{12}$ is base-point-free and birationally very ample inducing a curve $\ce$ on a smooth quadric $Q\subset\PP^3$. $\ce\in |\Oo_Q(6,6)|$
with two triple points and four nodes. However such a family arising this way is not sufficient enough to 
form a full component or a family with dimension at least $\Xx(16,15,5)$.

\end{itemize}
\end{rmk}

\vskip 12pt
 \vni{\bf Conflicts of Interest and Data Availability statement: }

 1. The author states that there is no conflict of interest. 
 
 2. This manuscript has no associated data.
\bibliographystyle{spmpsci} 

\begin{thebibliography}{111}
\bibitem{Accola1}
{R. Accola},
{\it Topics in the theory of Riemann surfaces.}
Lecture Notes in Mathematics 1595, Springer, Heidelberg, 1991.
\bibitem{AC2}
{E. Arbarello and M. Cornalba},
\textit{A few remarks about the variety of irreducible plane curves of given degree and genus.} Ann. Sci. \'Ec. Norm. Sup\'er. (4) \textbf{16} (1983), 467--483.
\bibitem{ACGH}
{E. Arbarello, M. Cornalba, P. Griffiths and J. Harris},
\textit{Geometry of Algebraic Curves Vol.I.}
Springer-Verlag, Berlin/Heidelberg/New York/Tokyo, 1985.
\bibitem{ACGH2}
{E. Arbarello, M. Cornalba and  P. Griffiths},
\textit{Geometry of Algebraic Curves Vol.II.}
Springer, Heidelberg, 2011.
\bibitem{B1}
{E. Ballico}, 
\textit{A remark on linear series on general $k$-gonal curves.} Boll. Un. Mat. Ital. A \textbf{(7) 3} (1989), no.2, 195 --197.
\bibitem{B}{E. Ballico},\textit{On the Clifford index of algebraic curves.} Proc. Amer. Math. Soc., \textbf{97} (1986), 217--218.
\bibitem{JPAA}
{E. Ballico, C. Fontanari and C. Keem}, 
\textit{On the Hilbert scheme of linearly normal curves in $\mathbb{P}^r$ of relatively high degree.}
J. Pure Appl. Algebra (2020) \textbf{224} (2020), 1115--1123.
\bibitem{bumi}%
{E. Ballico and C. Keem},
\textit{On the Hilbert scheme of smooth curves of degree $15$ and genus $14$ in $\PP^5$,} Bollettino UMI \textbf{17} no.3 (2024), 5537--557.
\bibitem{edinburgh}
{E. Ballico,  C. Keem}, 
\textit{On the Hilbert scheme of smooth curves of degree $15$ in $\PP^5$.} {\it To appear in The Proceeding of the Royal Society of Edinburgh Section A; Mathematics}, 
available at http://arxiv.org/abs/2310.00682. 
\bibitem{Beauville}
{A. Beauville},
\textit{Complex algebraic surfaces.}
Cambridge University Press, London/New York, 1983.
\bibitem{CKP}
{K. Cho, C. Keem and S. Park},
\textit{On the Hilbert scheme of trigonal curves and nearly extremal curves,} 
 Kyushu J. Math. \textbf{55} (2001), 1--12.
\bibitem{CC}
{C. Ciliberto},
\textit{On the Hilbert Scheme of Curves of Maximal Genus in a Projective Space.} 
Mathematische Zeitschrift \textbf{194} (1987), 451--463.
\bibitem{Coppens}{M. Coppens},\textit{Embeddings of general blowing-ups at points.} J. reine angew. Math. \textbf{469} (1995), 179--198.
\bibitem{CKM}
{M. Coppens, C. Keem and G. Martens},
\textit{The primitive length of a general $k$-gonal curve.} 
Indag. Math. (N.S.), \textbf{5} (1994), no. 2, 145--159.

\bibitem{DS}
{T. Dedieu and E. Sernesi},
\textit{Equigeneric and equisingular families of curves on surfaces}.
Publ. Mat. , \textbf{61} (2017), 175---212.


\bibitem{sandra}
{S. Di Rocco},
\textit{$k$-very ample line bundles on Del-Pezzo surfaces}.
Math. Nachr., \textbf{179} (1996), 47--56.
\bibitem{E1}
{L. Ein},
\textit{Hilbert scheme of smooth space curves}.
Ann. Scient. Ec. Norm. Sup. (4), \textbf{19} (1986), no. 4, 469--478.
\bibitem{E2}
{L. Ein},
{\it The irreducibility of the Hilbert scheme of complex space curves.}
Algebraic geometry, Bowdoin, 1985 (Brunswick, Maine, 1985), Proc. Sympos. Pure Math., 46, Part 1, Providence, RI: Amer. Math. Soc., 83--87.

\bibitem{H1} 
{J. Harris},
\textit{Curves in Projective space.}
in ``Sem. Math. Sup.,", Press Univ. Montr\'eal, Montr\'eal, 1982.
\bibitem{H2}
{J. Harris},
\textit{On the Severi problem.}
Invent. Math \textbf{84}, (1986), 445--461.
\bibitem{H3}
{J. Harris and I. Morrison},
\textit{Moduli of curves.}
Springer-Verlag, Berlin/Heidelberg/New York, 1991.
\bibitem{h1} {J. Harris}, \textit{Brill-Noether Theory,} Geometry of Riemann surfaces and their moduli spaces. Surveys in Differential Geometry Vol. XIV, Somerville, MA: Int. Press, pp 131-143, 2009.
\bibitem{I}{H. Iliev},\textit{On the irreducibility of the Hilbert scheme of space curves.} Proc. Amer. Math. Soc., \textbf{134} (2006), no. 10, 2823--2832.
\bibitem{I2}
{H. Iliev},
\textit{On the irreducibility of the {H}ilbert scheme of curves in
              {$\Bbb P^5$.}} Comm. Algebra., \textbf{36} (2008), no. 4, 1550--1564.
\bibitem{Keem}
{C. Keem}, 
\textit{Reducible Hilbert scheme of smooth curves with positive Brill-Noether number.}
Proc. Amer. Math. Soc., \textbf{122} (1994), no. 2, 349--354.
\bibitem{lengthy}
{C. Keem},
\textit{Existence and the reducibility of the Hilbert scheme of linearly normal
  curves in $\mathbb{P}^r$ of relatively high degrees,} J. Pure Appl. Algebra \textbf{227} (2023), 1115--1123, available at
https://arxiv.org/abs/2101.00559.


\bibitem{KKy1}{C. Keem and Y.-H. Kim}, \textit{Irreducibility of the Hilbert Scheme of smooth curves in $\PP^3$ of degree $g$ and genus $g$.} Arch. Math., \textbf{108} (2017), no. 6, 593--600.
\bibitem{KKy2}{C. Keem and Y.-H. Kim},\textit{Irreducibility of the Hilbert Scheme of smooth curves in $\PP^4$ of degree $g+2$ and genus $g$.} Arch. Math., \textbf{109} (2017), no. 6, 521--527.
\bibitem{KK3}{C. Keem and Y.-H. Kim},\textit{On the Hilbert scheme of linearly normal curves in $\mathbb{P}^4$ of degree $d = g+1$ and genus $g$.} Arch. Math., \textbf{113} (2019), no. 4, 373--384.


\bibitem{Nasu}
{H. Nasu}, 
\textit{The Hilbert scheme of space curves of degree $d$ and genus $3d-18$.}
Comm. Algebra \textbf{36} (2008), 4163-4185.

\bibitem{Sakai}
{M. Ohkouchi and F. Sakai}
\textit{The gonality of singular plane curves.}
Tokyo J. Math., \textbf{27} (2004), 137-147.
\bibitem{Sernesi}
{C. Ciliberto and E. Sernesi},
{\it Families of varieties and the Hilbert scheme.}
Lectures on Riemann surfaces (Trieste, 1987), 428--499, World Sci. Publ., Teaneck, NJ, 1989.

\bibitem{Sev}{F.  Severi}, \textit{Vorlesungen \"uber algebraische Geometrie.}Teubner, Leipzig, 1921.
\bibitem{tyomkin}
{I. Tyomkin},
\textit{On Severi varieties on Hirzebruch surfaces} 
Int. Math. Res. Not., \textbf{23} (2007), Art. ID rnm109.

\end{thebibliography}

\end{document}